\documentclass{amsart}

\usepackage{graphicx}
\usepackage{fancyhdr}
\newtheorem{theorem}{Theorem}[section]
\newtheorem{lemma}[theorem]{Lemma}

\theoremstyle{definition}
\newtheorem{definition}[theorem]{Definition}
\newtheorem{example}[theorem]{Example}

\theoremstyle{remark}
\newtheorem{remark}[theorem]{Remark}

\theoremstyle{proposition}
\newtheorem{proposition}[theorem]{Proposition}

\theoremstyle{corollary}
\newtheorem{corollary}[theorem]{Corollary}

\numberwithin{equation}{section}

\usepackage{indentfirst}
\usepackage{amsmath,amscd}
\usepackage{diagrams}
\usepackage{booktabs}
\usepackage[all]{xy}
\usepackage{setspace}
\usepackage{rotating}

\newcommand{\pf}{\noindent\begin {proof}}
\newcommand{\epf}{\end{proof}}

\newcommand{\Hom}{\mbox{\rm Hom}}

\def\Im{\mathop{\rm Im}\nolimits}
\def\Ker{\mathop{\rm Ker}\nolimits}
\def\Coker{\mathop{\rm Coker}\nolimits}

\def\mod{\mathop{\rm mod}\nolimits}

\def\add{\mathop{\rm add}\nolimits}

\def\Hom{\mathop{\rm Hom}\nolimits}

\def\lim{\mathop{\underrightarrow{\rm lim}}\nolimits}

\def\rad{\mathop{\rm rad}\nolimits}

\def\rad{\mathop{\rm rad}\nolimits}

\def\rad{\mathop{\rm rad}\nolimits}
\def\Det{\mathop{\rm Det}\nolimits}
\def\Soc{\mathop{\rm Soc}\nolimits}

\begin{document}

\title{Minimal right determiners of irreducible morphisms in string algebras}

\author{Xiaoxing Wu}
\address{}
\curraddr{}
\email{}
\thanks{}

\author{Zhaoyong Huang}
\address{Department of Mathematics, Nanjing University, Nanjing 210093, Jiangsu Province, P.R. China}
\email{wuxiaoxing1990@163.com; huangzy@nju.edu.cn}
\thanks{}


\date{}

\dedicatory{}


\begin{abstract}
Let $\Lambda$ be a finite dimensional string algebra over a field with the quiver $Q$ such that the underlying graph of $Q$ 
is a tree, and let $|\Det(\Lambda)|$ be the number of the minimal right determiners
of all irreducible morphisms between indecomposable left $\Lambda$-modules. Then we have
$$|\Det(\Lambda)|=2n-p-q-1,$$
where $n$ is the number of vertices in $Q$, $p=|\{i\mid i$ is a source in $Q$ with two neighbours$\}|$
and $q$ is the number of non-zero vertex ideals of $\Lambda$.
\end{abstract}

\maketitle

\setlength{\baselineskip}{14pt}

\section{Introduction}

In the seminal Philadelphia notes [2], Auslander introduced the notion of morphisms determined by objects,
which generalized that of almost split morphisms. However, until recently, this useful notion and related results
in [2] gained the deserved attention, see [7, 9--13].

Let $\Lambda$ be an artin algebra and $\mod \Lambda$ the category of finitely generated left $\Lambda$-modules.
Ringel corrected in [11, Theorem 1] a formula in [3, Theorem 2.6] for calculating a right determiner
of a morphism in $\mod\Lambda$; and then he reproved in [12, Theorem 3.4] a formula originally in [4]
for calculating the minimal right determiner of a morphism in $\mod \Lambda$. Based on these formulas,
we determined in [13, Theorems 3.13 and 3.15] the minimal right determiners of all irreducible morphisms
between indecomposable modules over a finite dimensional algebra of type $\mathbb{A}_n$.
We use $\Det(\Lambda)$ to denote the set of the minimal right determiners of all irreducible
morphisms between indecomposable modules in $\mod \Lambda$, and use $|\Det(\Lambda)|$ to denote the
cardinality of $\Det(\Lambda)$. In this paper, we continue the previous work mentioned above. We introduce the
notion of vertex ideals and prove the following

\begin{theorem}\label{1.1}
Let $\Lambda$ be a finite dimensional string algebra over a field with the quiver $Q$ such that the underlying graph of $Q$
is a tree. Then we have
$$|\Det(\Lambda)|=2n-p-q-1,$$
where $n$ is the number of vertices in $Q$, $p=|\{i\mid i$ is a source in $Q$ with two neighbours$\}|$
and $q$ is the number of non-zero vertex ideals of $\Lambda$.
\end{theorem}

We prove it in Section 3. Note that the proof of Theorem 1.1 is constructive, from which we can determine
the set $\Det(\Lambda)$.
In Section 4, we apply Theorem 1.1 to the case of algebras of Dynkin type; and in particular,
we obtain a unified version of [13, Theorems 3.13 and 3.15]. Finally, we give in Section 5 an example of
non-Dynkin type to illustrate this theorem.

\section{Preliminaries}

Throughout this paper, $\Lambda$ is a finite dimensional algebra over a field $K$
with the quiver $Q$, $\mod \Lambda$ is the category of finitely generated left $\Lambda$-modules
and $\tau$ is the Auslander-Reiten translation.
For an arrow $\alpha$ in $Q$, $s(\alpha)$ and $e(\alpha)$ are the starting and end points of $\alpha$, respectively.
We use $P(i)$, $I(i)$ and $S(i)$ to denote the indecomposable projective, injective
and simple modules corresponding to the vertex $i$, respectively. For a module $M$ in $\mod \Lambda$, we use
$\Soc(M)$ and $\add_{\Lambda}M$ to denote the socle of $M$ and the full subcategory of $\mod \Lambda$ consisting
of direct summands of finite direct sums of copies of $M$, respectively. For a set $S$, we use $|S|$ to
denote the cardinality of $S$.

The original definition of morphisms determined by objects in [2] is based on the notion
of subfunctors determined by objects. However, in the relevant papers, ones prefer the following definition
since it is easier to understand.

\begin{definition}\label{2.1} {\rm ([11, 12])}
For a module $C\in \mod \Lambda$, a morphism $f\in \Hom_{\Lambda}(X,Y)$ is said to be {\bf right determined}
by $C$ (simply {\bf $C$-right determined})
if the following condition is satisfied: given for any $f'\in\Hom_{\Lambda}(X',Y)$ such that $f'\phi$
factors through $f$ for all $\phi\in\Hom_{\Lambda}(C,X')$, then $f'$ factors through $f$; that is,
in the following diagram, if there exists $\phi'\in\Hom_{\Lambda}(C,X)$ such that $f'\phi=f\phi'$,
then there exists $h\in\Hom_{\Lambda}(X',X)$ such that $f'=fh$.
$$\xymatrix{C \ar[r]^{\phi} \ar@{=}[d] & X' \ar[r]^{f'} \ar@{-->}[d]^{h} & Y \ar@{=}[d] \\
C \ar@{-->}[r]^{\phi'} & X \ar[r]^{f} & Y.}$$
In this case, $C$ is called a {\bf right determiner} of $f$.
\end{definition}

\begin{definition}\label{2.2} ([11, p.984])
Given a morphism $f\in\Hom_{\Lambda}(B,C)$ with $B=B_1\oplus B_2$ such that $B_1 \subseteq \Ker f$
and $f|_{B_2}$ is right minimal, then we call $\Ker f|_{B_2}$ the {\bf intrinsic kernel} of $f$.
\end{definition}

\begin{definition}\label{2.3} {\rm ([12, p.418])}
An indecomposable projective module $P\in \mod \Lambda$ is said to {\bf almost factor through}
$f\in \Hom_{\Lambda}(M,N)$ provided that there exists a commutative diagram of the following form
$$\xymatrix{
\rad P \ar[d] \ar[r]^{i}
& P \ar[d]^{h}  \\
M \ar[r]^{f} & N,}
$$
where $i$ is the inclusion map and $\rad P$ is the radical of $P$, such that $\Im h$ is not contained in $\Im f$.
\end{definition}

The following is the determiner formula.

\begin{theorem}\label{2.4}
Let $f$ be a morphism in $\mod \Lambda$. Let $C(f)$ be the direct sum of the indecomposable
modules of the form $\tau^{-1}K$, where $K$ is an indecomposable direct summand of
the intrinsic kernel of $f$ and of the indecomposable projective modules which
almost factor through $f$, one from each isomorphism class. Then we have
\begin{enumerate}
\item[(1)] {\rm ([12, Theorem 3.4], [11, Theorem 2] and [4, Corollary XI.2.3])} $f$ is right $C$-determined if and only if $C(f)\in\add_{\Lambda}C$.
\item[(2)] {\rm ([13, Theorem 2.4(2)])} If $f$ is irreducible, then $C(f)$=$\tau^{-1}\Ker f\oplus(\oplus P_i)$,
where all $P_i$ are pairwise non-isomorphic indecomposable projective modules almost factoring through $f$.
\end{enumerate}
\end{theorem}

The first assertion in this theorem suggests to call $C(f)$ the {\bf minimal right determiner} of $f$ ([11, 12]).
We use $\Det(\Lambda)$ to denote the set of the minimal right determiners of all irreducible
morphisms between indecomposable modules in $\mod \Lambda$.

We use $Q_0:=\{1,\cdots,n\}$ and $Q_1$ to denote the set of vertices and the set of arrows in $Q$, respectively.

\begin{definition} {\rm ([5, p.534] and [6, p.157])}
Let $\Lambda=KQ/I$ with $I$ an ideal of $KQ$. Then $\Lambda$ is called a \textbf{special biserial algebra} provided the following
conditions are satisfied.
\begin{enumerate}
\item[(1)] For each $i\in Q_0$, we have $|\{\alpha\in Q_1\mid s(\alpha)=i\}| \leq 2$ and $|\{\alpha\in Q_1\mid e(\alpha)=i\}|\leq 2$.
\item[(2)] For $\alpha,\beta,\gamma\in Q_1$ with $e(\alpha)=e(\beta)=s(\gamma)$ and $\alpha \neq \beta$,
we have $\gamma\alpha \in I$ or $\gamma\beta \in I$.
\item[(3)] For $\alpha,\beta,\gamma\in Q_1$ with $s(\alpha)=s(\beta)=e(\gamma)$ and $\alpha \neq \beta$,
we have $\alpha\gamma \in I$ or $\beta\gamma \in I$.
\end{enumerate}
A special biserial algebra is called a \textbf{string algebra} if the following condition is satisfied.
\begin{enumerate}
\item[(4)] The ideal $I$ can be generated by zero relations.
\end{enumerate}
\end{definition}

\section{Representation-finite string algebras}

In this section, $\Lambda$ is a string algebra with the quiver $Q$ such that the underlying graph of $Q$
is a tree. Then either $\Lambda=KQ$ or $\Lambda=KQ/I$ with $I$ an admissible ideal of $KQ$.
By the definition of string algebras, the former case occurs only if $\Lambda$ is of type $\mathbb{A}_n$.
In either case, $\Lambda$ is of finite representation type ([6, p.161, Theorem] or [8, Theorem 1.2(2)]), and we may assume
that there are $n$ vertices and $n-1$
arrows in $Q$. All morphisms considered are irreducible morphisms between indecomposable modules in $\mod \Lambda$.

By [6, p.147], there are only two types of almost split sequences in
$\mod\Lambda$, that is, the middle term in an almost split sequence is indecomposable or is a direct
sum of two indecomposable modules:
$$0\to L \to M \to N \to 0\eqno{(3.1)}$$
and
$$0\to L \to M_1\oplus M_2 \to N \to 0.\eqno{(3.2)}$$

\begin{proposition}\label{3.1} {\rm([6, p.174, Corollary])}
\begin{enumerate}
\item[(1)] The only almost split sequences in $\mod \Lambda$ of type (3.1) are those of the form
$$0 \to U(\beta) \to N(\beta) \to V(\beta) \to 0$$
with $\beta$ an arrow in $Q$.
\item[(2)]
The number of almost split sequences of type (3.1) in $\mod \Lambda$ is $n-1$.
\end{enumerate}
\end{proposition}

We give the following useful remark.

\begin{remark}\label{3.2}
In [13] the algebra $\Lambda$ is assumed to be of type $\mathbb{A}_n$. We point out that
all results from 3.1 to 3.8 in [13] hold true in the setting of this paper even without changing the proofs there.
To avoid repeating, we will not list these results in details here, but cite them directly when needed.
\end{remark}

Note that an irreducible morphism is either a proper monomorphism or a proper
epimorphism. By [13, Corollary 3.7], we have
\begin{align*}
&\ \ \ \ \ \{C(f)\mid f\ \text{is an epic irreducible morphism in}\ \mod \Lambda\}\\
& =\{\text{the last terms in almost split sequences of type (3.1) as in Proposition 3.1}\},
\end{align*}
and its cardinality is $n-1$. So, in the following, we only need to determine the minimal right determiners
of all irreducible monomorphisms.

For $\alpha\in Q_1$, recall from [1, p.43] that $s(\alpha)$ and $e(\alpha)$ are called the {\bf neighbours} of $e(\alpha)$ and $s(\alpha)$,
respectively. By the definition of string algebras, we can give a complete classification of the vertices in $Q$ as follows.

\begin{enumerate}
\item[(v1)] The vertex $i_1$ with a unique neighbour:
$$\xymatrix{i_1 \ar[r] & \cdots &&&&\ \ \ {\rm (v1.1)}}$$
and
$$\xymatrix{i_1 & \cdots. \ar[l] &&&&\ \ \ {\rm (v1.2)}}$$
\item[(v2)] The vertex $i_2$ with two neighbours:
$$\xymatrix{\cdots & i_2 \ar[l] \ar[r] & \cdots, &&& {\rm (v2.1)}}$$
$$\xymatrix{\cdots \ar[r] & i_2 & \cdots \ar[l] &&&\ \ {\rm (v2.2)}}$$
and
$$\xymatrix{\cdots \ar[r] & i_2 \ar[r]& \cdots. &&&\ \ {\rm (v2.3)}}$$
\item[(v3)] The vertex $i_3$ with three neighbours:
$$
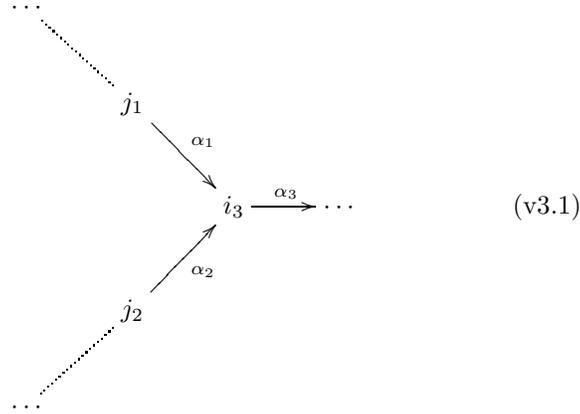
\xymatrix{\cdots \ar@{.}[dr] &&\\
& j_1 \ar[dr]^{\alpha_1} &&\\
& & i_3 \ar[r]^{\alpha_3} & \cdots && {\rm (v3.1)}\\
& j_2  \ar[ur]_{\alpha_2} &&\\
\cdots \ar@{.}[ur] &&\\}$$
such that at least one in $\{\alpha_3\alpha_1,\; \alpha_3\alpha_2\}$ is in $I$; and
$$
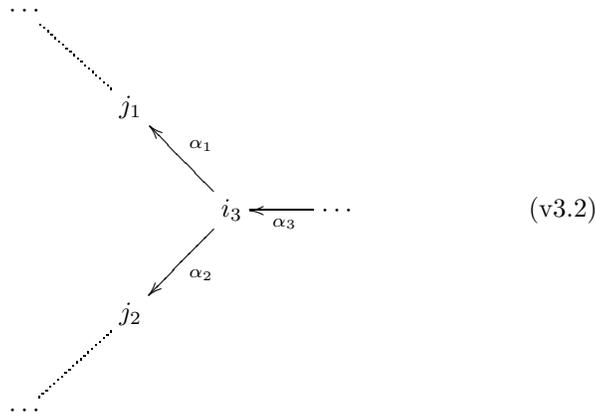
\xymatrix{\cdots \ar@{.}[dr] &&\\
& j_1 &&\\
& & i_3 \ar[ul]_{\alpha_1}\ar[dl]^{\alpha_2} & \cdots\ar[l]^{\alpha_3} && \ \ {\rm (v3.2)}\\
& j_2  &&\\
\cdots \ar@{.}[ur] &&\\}$$
such that at least one in $\{\alpha_1\alpha_3,\; \alpha_2\alpha_3\}$ is in $I$.
\item[(v4)] The vertex $i_4$ with four neighbours:
$$
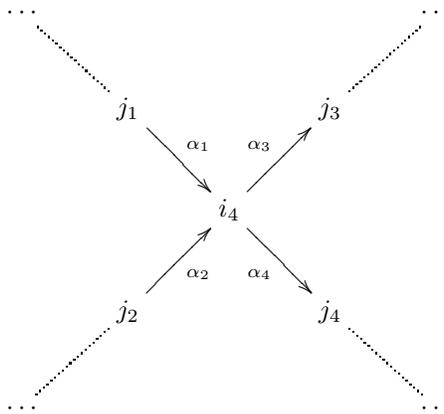
\xymatrix{
\cdots \ar@{.}[dr] &&&& \cdots \ar@{.}[dl] \\
& j_1  \ar[dr]^{\alpha_1} & & j_3 \\
& & i_4 \ar[ur]^{\alpha_3} \ar[dr]_{\alpha_4}\\
& j_2 \ar[ur]_{\alpha_2} & & j_4 \\
\cdots \ar@{.}[ur] &&&& \cdots \ar@{.}[ul] \\
}$$
such that at least one of the two sets $\{\alpha_3\alpha_1,\; \alpha_4\alpha_2\}$
and $\{\alpha_4\alpha_1,\; \alpha_3\alpha_2\}$ is in $I$.
\end{enumerate}

For convenience sake, we denote the subquivers in (v3) with 4 vertices including the vertex
$i_3$ and its 3 neighbours by $X_{i_3}$, and denote the subquivers in (v4) with 5 vertices including the vertex
$i_4$ and its 4 neighbours by $X_{i_4}$. A full subquiver of $Q$ between two vertices $i$ and $j$ is denoted by
$<i,j>$. For a subquiver $Q'$ of $Q$, we write $I|_{Q'}:=I\cap KQ'$. The following definition
is crucial in the sequel.

\begin{definition}\label{3.3}
For the vertex $i$ in $Q$, we define the {\bf vertex ideal} $J_i$ of $\Lambda$ according to the above classification of vertices as follows.
\begin{enumerate}
\item[(1)] For the sink $i_1$ of type (v1.2), define
$$J_{i_1}=\begin{cases}
0, &\mbox{if there exists $j\in Q_0$ such that $|\{\alpha\in Q_1\mid s(\alpha)=j\}|=2$,}\\
&\mbox{$<j,i_1>$ is linear and $I|_{<j,i_1>}=0$;}\\
\Lambda, &\mbox{if $\Lambda$ is a path algebra with a unique sink $i_1$};\\
I, &\mbox{otherwise.}
\end{cases}$$
\item[(2)] For the sink $i_2$ of type (v2.2), define
$$J_{i_2}=\begin{cases}
0, &\mbox{if there exists $j\in Q_0$ such that $|\{\alpha\in Q_1\mid s(\alpha)=j\}|=2$, }\\
&\mbox{$<j,i_2>$ is linear and $I|_{<j,i_2>}=0$;}\\
\Lambda, &\mbox{if $\Lambda$ is a path algebra with a unique sink $i_2$};\\
I, &\mbox{otherwise.}
\end{cases}$$
\item[(3)] For the vertex $i_3$ of type (v3), define
$$J_{i_3}=
\begin{cases}
0, \begin{cases} &\mbox{(a) if $i_3$ is of type (v3.1); or} \\
&\mbox{(b) if $i_3$ is of type (v3.2) and there exists $j\in Q_0$ such that }\\
&\mbox{\ \ \ \ \ $|\{\alpha\in Q_1\mid s(\alpha)=j\}|=2$, $<j,i_3>$ is linear,
$I|_{<j,i_3>}=0$, }\\
&\mbox{\ \ \ \ \ $I|_{<j,j_1>} \neq 0$ and $I|_{<j,j_2>} \neq 0$;}
\end{cases} \\
I|_{X_{i_3}},\ \ \mbox{otherwise.}
\end{cases}$$
\item[(4)]For the vertex $i_4$ of type (v4), define
$$J_{i_4}=\begin{cases}
0, &\mbox{if there exists $j\in Q_0$ such that $|\{\alpha\in Q_1\mid s(\alpha)=j\}|=2$,}\\
&\mbox{$<j,i_4>$ is linear, $I|_{<j,i_4>}=0$, $I|_{<j,j_3>}\neq0$
and $I|_{<j,j_4>}\neq 0$;}\\
I|_{X_{i_4}},   &\mbox{otherwise.}
\end{cases}$$
\end{enumerate}
\end{definition}

By [13, Corollary 3.2], we have that the minimal right determiner of any irreducible monomorphism is
indecomposable projective. The following lemma gives some criteria for judging when an indecomposable projective
module is the minimal right determiner of an irreducible monomorphism.

\begin{lemma}\label{3.4}
For an irreducible monomorphism $f$ in $\mod \Lambda$, the following statements are equivalent.
\begin{enumerate}
\item[(1)] $P(i)=C(f)$.
\item[(2)] There exists an irreducible monomorphism $f_1:X \rightarrow P(j)$ with $X$ indecomposable
such that $C(f_1)=P(i)=C(f)$.
\item[(3)] $P(i)$ almost factors through $f$.
\item[(4)] $S(i)=\Soc (\Coker f)$.
\end{enumerate}
If one of the above equivalent conditions is satisfied and $j$ is as in (2), then
the subquiver $<j,i>$ is linear and $I|_{<j,i>}=0$.
\end{lemma}

\begin{proof}
By [13, Theorem 3.5(1)], we have $(1)\Leftrightarrow(2)$. By Theorem 2.4(2) and [13, Remark 3.3],
we have $(1)\Leftrightarrow(3)$. By [13, Corollary 3.4(1)], we have $(1)\Leftrightarrow(4)$.

If one of the above equivalent conditions is satisfied and $j$ is as in (2),
then $P(i)$ almost factors through $f_1$ and $\Hom_{\Lambda}(P(i), P(j)) \neq 0$, which indicates that
there exists a non-zero path from $j$ to $i$ in $Q$, that is, $<j,i>$
is linear and $I|_{<j,i>}=0$.
\end{proof}

We need some further preparation.

\begin{lemma}\label{3.5}
For a vertex $i\in Q_0$, we have
\begin{enumerate}
\item[(1)] $|\{\alpha\in Q_1\mid s(\alpha)=i\}|=1$ if and only if $\rad P(i)$ is indecomposable. In this case,
for any irreducible monomorphism $X \rightarrow P(i)$ with $X$ indecomposable, we have $X\cong \rad P(i)$.
\item[(2)] $|\{\alpha\in Q_1\mid s(\alpha)=i\}|=2$ if and only if $\rad P(i)=M_i \oplus N_i$ with
$M_i$ and $N_i$ indecomposable. In this case, for any irreducible monomorphism $X \rightarrow P(i)$ with
$X$ indecomposable, there exists an indecomposable module $Y\in\mod \Lambda$, such that $\rad P(i)\cong X\oplus Y$.
\end{enumerate}
\end{lemma}

\begin{proof}
The first assertions in (1) and (2) are well known. If $X \rightarrow P(i)$ is an irreducible monomorphism,
then $X$ is isomorphic to a submodule of the unique maximal submodule
$\rad P(i)$ of $P(i)$, and hence isomorphic to a direct summand of $\rad P(i)$ by [4, Lemma V.5.1(b)].
\end{proof}

The following three lemmas are useful.

\begin{lemma}\label{3.6}
Let $i\in Q_0$ with $|\{\alpha\in Q_1\mid s(\alpha)=i\}|=1$. Then $P(i)=C(\rad P(i)\hookrightarrow P(i))$.
\end{lemma}

\begin{proof}
It follows from Lemmas 3.5(1) and 3.4.
\end{proof}

\begin{lemma}\label{3.7}
Let $i\in Q_0$ with $|\{\alpha\in Q_1\mid s(\alpha)=i\}|\neq 1$ and $P(i)\in \Det(\Lambda)$.
\begin{enumerate}
\item[(1)] If $j$ is as in Lemma 3.4(2), then $|\{\alpha\in Q_1\mid s(\alpha)=j\}|=2$.
\item[(2)] If the vertex $i$ is a sink $i_1$ of type (v1.2) (resp. a sink $i_2$ of type (v2.2)),
then $J_{i_1}=0$ (resp. $J_{i_2}=0$).
\end{enumerate}
\end{lemma}

\begin{proof}
(1) Let $i\in Q_0$ with $|\{\alpha\in Q_1\mid s(\alpha)=j\}|\neq 1$ and $P(i)\in \Det(\Lambda)$.
If $j$ is as in Lemma 3.4(2), then the subquiver $<j,i>$ is linear and $I|_{<j,i>}=0$ by Lemma 3.4.
If $|\{\alpha\in Q_1\mid s(\alpha)=j\}|=0$ (that is, $j$ is a sink), then $\rad P(j)=0$, and hence $f_1=0$ by Lemma 3.5(1),
a contradiction. If $|\{\alpha\in Q_1\mid s(\alpha)=j\}|=1$,
then $C(f_1)=P(j)$ by Lemmas 3.5(1) and 3.6. It is clear that $i\neq j$, so we have $C(f_1)\neq P(i)$, also a contradiction.
Consequently we conclude that $|\{\alpha\in Q_1\mid s(\alpha)=j\}|=2$.

(2) If the vertex $i$ is a sink $i_1$ of type (v1.2) (resp. a sink $i_2$ of type (v2.2)),
then $J_{i_1}=0$ (resp. $J_{i_2}=0$) by (1) and the definition of vertex ideals.
\end{proof}

\begin{lemma}\label{3.8}
Let $i\in Q_0$ with $|\{\alpha\in Q_1\mid s(\alpha)=i\}|=2$. Then $P(i)\neq C(f)$
for any irreducible monomorphism $f:X\to P(i)$.
\end{lemma}

\begin{proof}
Let $i\in Q_0$ with $|\{\alpha\in Q_1\mid s(\alpha)=i\}|=2$. Suppose $P(i)=C(f)$ for some irreducible monomorphism
$f:X\to P(i)$. It is clear that $f$ can be assumed to be an inclusion. By Lemma 3.5(2), we have that
$\rad P(i)=M_i \oplus N_i$ with $M_i$ and $N_i$ indecomposable and that either $X=M_{i}$ or $X=N_{i}$.
Consider the following diagram:
$$\xymatrix@-12pt{
\cdots \ar@{-}[r] &0\ar@{.>}[d] & \overset{i}{K}\ar[l]\ar[r]\ar@{.>}[d]^{1_K} &0 \ar@{.>}[d] \ar@{-}[r] &\cdots\\
\cdots \ar@{-}[r] &K & K\ar[l]_{1_K}\ar[r] &0 \ar@{-}[r] &\cdots \\}$$
or
$$\xymatrix@-12pt{\cdots \ar@{-}[r] &0\ar@{.>}[d] & \overset{i}{K}\ar[l]\ar[r]\ar@{.>}[d]^{1_K} &0 \ar@{.>}[d] \ar@{-}[r] &\cdots\\
\cdots \ar@{-}[r] &0 & K\ar[l]\ar[r]^{1_K} &K\ar@{-}[r] &\cdots. \\}$$
In either diagram, the above is part of the representation of $S(i)$ around $i$ and the below is part of the representation
of $\Coker f$ around $i$. Notice that neither the left square in the first diagram nor the right
square in the second diagram is commutative, so $S(i)$ is not a submodule of $\Coker f$.
It implies $S(i)\neq\Soc(\Coker f)$, which contradicts Lemma 3.4. The assertion follows.
\end{proof}

The following is a key step toward proving the main result.

\begin{theorem}\label{3.9}
For a vertex $i_k$ in $Q$, $P(i_k)\in \Det(\Lambda)$
if and only if $i_k$ is one of the following types.
\begin{enumerate}
\item[(1)]
\begin{enumerate}
\item[(1.1)] a source $i_1$ of type (v1.1);
\item[(1.2)] a sink $i_1$ of type (v1.2) and $J_{i_1}=0$.
\end{enumerate}
\item[(2)]
\begin{enumerate}
\item[(2.1)] a sink $i_2$ of type (v2.2) and $J_{i_2}=0$;
\item[(2.2)] $i_2$ of type (v2.3).
\end{enumerate}
\item[(3)]
\begin{enumerate}
\item[(3.1)] $i_3$ of type (v3.1);
\item[(3.2)] $i_3$ of type (v3.2) and $J_{i_3}=0$.
\end{enumerate}
\item[(4)] $i_4$ of type (v4) and $J_{i_4}=0$.
\end{enumerate}
\end{theorem}

\begin{proof}
(1) If $i_1$ is a source of type (v1.1), then $P(i_1)=C(\rad P(i_1) \hookrightarrow P(i_1))$ by Lemma 3.6.

Let $i_1$ be a sink of type (v1.2). If $J_{i_1}=0$, then there exists $j\in Q_0$ such that $|\{\alpha\in Q_1\mid s(\alpha)=j\}|=2$,
$<j,i_1>$ is linear and $I|_{<j,i_1>}=0$. By Lemma 3.5(2), we have $\rad P(j)=M_j \oplus N_j$ with
$M_j$ and $N_j$ indecomposable. So there exists a subquiver of the Auslander-Reiten quiver of $\mod \Lambda$ as follows.
$$
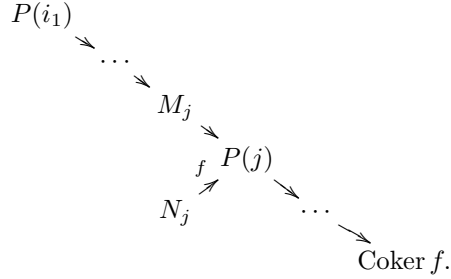
\xymatrix@-20pt{
P(i_1) \ar[dr] \\
& \cdots \ar[dr] \\
& & M_j \ar[dr] \\
& & &P(j)\ar[dr] \\
& & N_j \ar[ur]^{f} & &\cdots\ar[dr] \\
&&&&& \Coker f.}$$
It is straightforward to calculate that $\Soc(\Coker f)=S(i_1)$. So $P(i_1)=C(f)$ by Lemma 3.4.
Conversely, if $P(i_1)\in\Det(\Lambda)$, then $J_{i_1}=0$ by Lemma 3.7(2).

(2) Let $i_2$ be a source of type (v2.1). If $P(i_2)\in\Det(\Lambda)$,
then by Lemmas 3.4 and 3.7(1), there exists an irreducible monomorphism $f_1:X \rightarrow P(j)$ with $X$ indecomposable
such that $P(i_2)=C(f_1)$, $|\{\alpha\in Q_1\mid s(\alpha)=j\}|=2$, the subquiver $<j,i_2>$
is linear and $I|_{<j,i_2>}=0$, and hence $j=i_2$. It contradicts Lemma 3.8. Thus we have $P(i_2) \notin \Det(\Lambda)$.

Let $i_2$ be a sink of type (v2.2). If $J_{i_2}=0$, then there exists $j\in Q_0$ such that $|\{\alpha\in Q_1\mid s(\alpha)=j\}|=2$,
$<j,i_2>$ is linear and $I|_{<j,i_2>}=0$. By Lemma 3.5(2), we have $\rad P(j)=M_j \oplus N_j$ with
$M_j$ and $N_j$ indecomposable. So there exists a subquiver of the Auslander-Reiten quiver of
$\mod \Lambda$ as follows.
$$
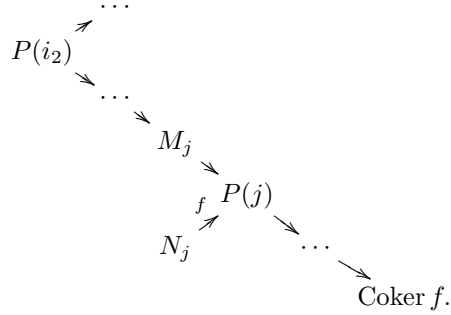
\xymatrix@-20pt{
& \cdots  \\
P(i_2) \ar[dr] \ar[ur]  \\
& \cdots \ar[dr] \\
& & M_j \ar[dr] \\
& & &P(j)\ar[dr]  \\
& & N_j \ar[ur]^{f} & &\cdots\ar[dr]  \\
&&&&& \Coker f.}$$
It is straightforward to calculate that $\Soc(\Coker f)=S(i_2)$. So $P(i_2)=C(f)$ by Lemma 3.4.
Conversely, if $P(i_2)\in\Det(\Lambda)$, then $J_{i_2}=0$ by Lemma 3.7(2).

If $i_2$ is of type (v2.3), then $P(i_2)=C(\rad P(i_2) \hookrightarrow P(i_2))$ by Lemma 3.6.

(3) If $i_3$ is of type (v3.1), then $P(i_3)=C(\rad P(i_3) \hookrightarrow P(i_3))$ by Lemma 3.6 again.

Let $i_3$ be of type (v3.2). If $J_{i_3}=0$, then there exists $j\in Q_0$ such that
$|\{\alpha\in Q_1\mid s(\alpha)=j\}|=2$, $<j,i_3>$ is linear,
$I|_{<j,i_3>}=0$, $I|_{<j,j_3>} \neq 0$ and $I|_{<j,k_3>} \neq 0$.
By Lemma 3.5(2), we have $\rad P(j)=M_j \oplus N_j$ with $M_j$ and $N_j$ indecomposable.
So there exists a subquiver of the Auslander-Reiten quiver of $\mod \Lambda$ as follows.
$$
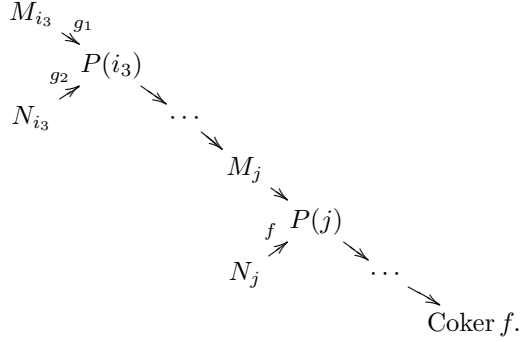
\xymatrix@-20pt{
  M_{i_3}    \ar[dr]^{g_1}                                      \\
& P(i_3) \ar[dr]                                         \\
  N_{i_3}    \ar[ur]^{g_2} & & \cdots \ar[dr]                    \\
&                &                & M_j \ar[dr]      \\
&                &                &            &P(j)\ar[dr] \\
&                &                & N_j \ar[ur]^{f} && \cdots \ar[dr] \\
&&&&&& \Coker f.
}$$
It is straightforward to calculate that $\Soc(\Coker f)=S(i_3)$. So $P(i_3)=C(f)$ by Lemma 3.4.

Conversely, if $P(i_3) \in \Det(\Lambda)$, then by Lemma 3.4, there exists an irreducible monomorphism $f_1:X \rightarrow P(j)$
with $X$ indecomposable such that $C(f_1)=P(i_3)$, the subquiver $<j,i_3>$
is linear and $I|_{<j,i_3>}=0$. By Lemmas 3.8 and 3.7(1), we have $j\neq i_3$ and
$|\{\alpha\in Q_1\mid s(\alpha)=j\}|=2$. Parts of the representations of $S(i_3)$ and $\Coker f$ around $i_3$ are shown as below.
$$
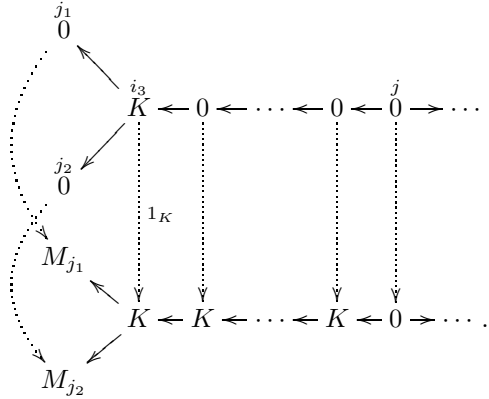
\xymatrix@-15pt{
\overset{j_1}{0}\ar@{.>}@/_7mm/[ddd] \\
& \overset{i_3}{K} \ar[ul]\ar[dl]\ar@{.>}[ddd]^{1_K} & 0\ar[l]\ar@{.>}[ddd] &\cdots\ar[l] &0\ar[l]\ar@{.>}[ddd] &\overset{j}{0}\ar[l]\ar[r]\ar@{.>}[ddd] &\cdots \\
\overset{j_2}{0}\ar@{.>}@/_7mm/[ddd] \\
M_{j_1} \\
& K \ar[ul]\ar[dl]& K\ar[l]&\cdots\ar[l] &K\ar[l]&0\ar[l]\ar[r]&\cdots. \\
M_{j_2} \\}$$
Because $S(i_3)=\Soc(\Coker f)$ by Lemma 3.4, the above diagram is commutative. It implies $M_{j_1}=0=M_{j_2}$.
So $I|_{<j,j_1>} \neq 0$ and $I|_{<j,j_2>} \neq 0$. Thus we have $J_{i_3}=0$ by the definition of vertex ideals.

(4) Let $i_4$ be of type (v4). If $J_{i_4}=0$, then there exists $j\in Q_0$ such that $|\{\alpha\in Q_1\mid s(\alpha)=j\}|=2$,
$<j,i_4>$ is linear, $I|_{<j,i_4>}=0$, $I|_{<j,j_3>}\neq0$
and $I|_{<j,j_4>}\neq 0$. By Lemma 3.5(2), we have $\rad P(j)=M_j \oplus N_j$ with
$M_j$ and $N_j$ indecomposable. So there exists a subquiver of the Auslander-Reiten quiver of $\mod \Lambda$ as follows.
$$
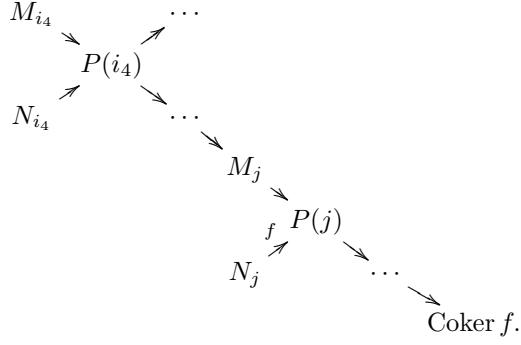
\xymatrix@-20pt{
M_{i_4} \ar[dr]& & \cdots \\
& P(i_4) \ar[dr] \ar[ur] \\
N_{i_4} \ar[ur]& & \cdots \ar[dr]  \\
&  &  & M_j \ar[dr] \\
& & &  &P(j)\ar[dr] \\
& &  & N_j \ar[ur]^{f} & &\cdots \ar[dr] \\
&&&&&& \Coker f.
}$$
It is straightforward to calculate that $\Soc(\Coker f)=S(i_4)$. So $P(i_4)=C(f)$ by Lemma 3.4.

Conversely, if $P(i_4) \in \Det(\Lambda)$, then by Lemma 3.4, there exists an irreducible monomorphism $f_1:X \rightarrow P(j)$
with $X$ indecomposable such that $C(f_1)=P(i_4)$, the subquiver $<j,i_4>$
is linear and $I|_{<j,i_4>}=0$. By Lemmas 3.8 and 3.7(1), we have $j\neq i_4$
and $|\{\alpha\in Q_1\mid s(\alpha)=j\}|=2$. Parts of the representations of $S(i_4)$ and $\Coker f$ around $i_4$ are shown as below.
$$
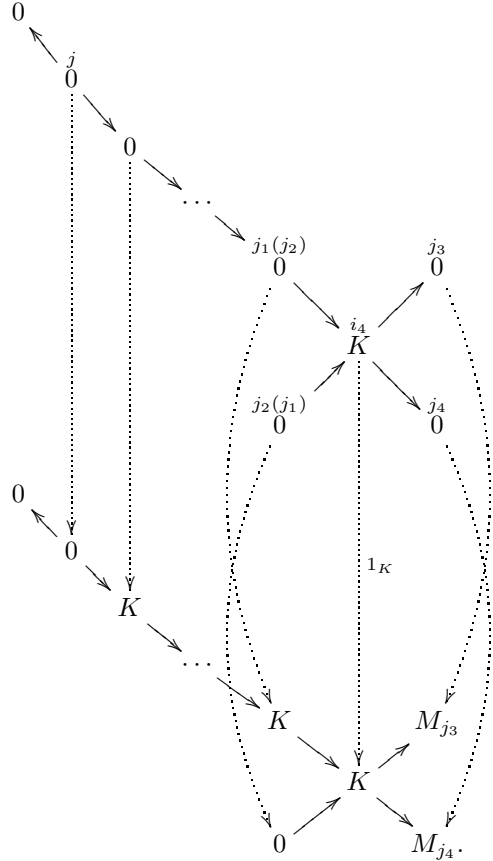
\xymatrix@-15pt{0 \\
&\overset{j}{0}\ar@{.>}[ddddddd]\ar[ul]\ar[dr]\\
&&0\ar@{.>}[ddddddd]\ar[dr]\\
&&&\cdots\ar[dr]\\
&&&&\overset{j_1(j_2)}{0}\ar@/_7mm/@{.>}[ddddddd]\ar[dr] &&\overset{j_3}{0}\ar@/^7mm/@{.>}[ddddddd]\\
&&&&&\overset{i_4}{K} \ar@{.>}[ddddddd]^{1_K}\ar[ur]\ar[dr]\\
&&&&\overset{j_2(j_1)}{0}\ar@/_7mm/@{.>}[ddddddd]\ar[ur] &&\overset{j_4}{0}\ar@/^7mm/@{.>}[ddddddd] \\
0 \\
&0\ar[ul]\ar[dr]\\
&&K\ar[dr]\\
&&&\cdots\ar[dr]\\
&&&&K\ar[dr] &&M_{j_3}\\
&&&&&K \ar[ur]\ar[dr]\\
&&&&0\ar[ur] &&M_{j_4}. \\}$$
Because $S(i_4)=\Soc(\Coker f)$ by Lemma 3.4, the above diagram is commutative. It implies $M_{j_3}=0=M_{j_4}$.
So $I|_{<j,j_3>} \neq 0$ and $I|_{<j,j_4>} \neq 0$. Thus we have $J_{i_4}=0$ by the definition of vertex ideals.
\end{proof}

We are now in a position to give the the main result in this paper.

\begin{theorem}\label{3.10}
Set
$$p:=|\{i\mid i\ \text{is a source of type (v2.1)}\}|,$$
$$q:=|\{J_{i_j}\neq 0 \mid 1\leq j\leq 4\}|.$$
Then we have
$$|\Det(\Lambda)|=2n-p-q-1.$$
\end{theorem}

\begin{proof}
By [13, Corollary 3.7], we have that the number of the (non-projective) minimal right determiners of all
irreducible epimorphisms is $n-1$. By Theorem 3.9, we have that the number of the (projective) minimal right
determiners of all irreducible monomorphisms is $n-p-q$. So we have
$$|\Det(\Lambda)|=(n-1)+(n-p-q)=2n-p-q-1.$$
\end{proof}

The following two results show that the distribution of the projective minimal right determiners
can determine the orientation of a quiver in some cases. The first one is a generalization of
[13, Corollary 3.12].

\begin{proposition}\label{3.11}
Assume that there are no vertices of type (v4) in $Q$ and $j\in Q_0$ is a sink of type (v1.2).
Then the following statements are equivalent.
\begin{enumerate}
\item[(1)] The projective minimal right determiners are $\{P(i)\mid 1 \leq i \leq n$ but $i\neq j\}$.
\item[(2)] $j$ is the unique sink in $Q$.
\end{enumerate}
\end{proposition}

\begin{proof}
$(2)\Rightarrow (1)$ Assume that $j$ is the unique sink in $Q$.
Let $i\in Q_0$ with $i\neq j$. Then $i$ is one of the following types: (v1.1), (v2.3), (v3.1). Now the assertion
follows from Theorem 3.9.

$(1)\Rightarrow (2)$ Assume that the projective minimal right determiners are $\{P(i)\mid 1 \leq i \leq n$ but $i\neq j\}$.
Let $i\in Q_0$ with $i\neq j$. Because $P(i)\in\Det(\Lambda)$ by (1), we have that $i$ is not of type (v2.1)
by Theorem 3.9. So, to show that $i$ is not a sink, it suffices to show that $i$ is not of any one of
the following types: (v1.2), (v2.2), (v3.2). Assume that $i$ is of one of these three types.
It follows from Lemmas 3.4 and 3.8 that
there exists $k_1\in Q_0$ with $k_1\neq i$ such that $|\{\alpha\in Q_1\mid s(\alpha)=k_1\}|=2$ and $<k_1,i>$ is linear.
It is clear that $k_1\neq j$. So $P(k_1)\in \Det(\Lambda)$ and $k_1$ is of type either (v2.1) or (v3.2).

If $k_1$ is of type (v2.1), then by Theorem 3.9, we have $P(k_1)\notin \Det(\Lambda)$, a contradiction.
If $k_1$ is of type (v3.2), then by Lemmas 3.4 and 3.8 again, there exists $k_2\in Q_0$ with $k_2\neq k_1$
such that $|\{\alpha\in Q_1\mid s(\alpha)=k_2\}|=2$ and $<k_1,k_2>$ is linear. It is clear that $k_2\neq j$.
So $P(k_2)\in \Det(\Lambda)$ and $k_2$ is of type either (v2.1) or (v3.2). By the same reason as above,
we have that $k_2$ is of type (v3.2) but not of type (v2.1). Note that the quiver $Q$ is acyclic. So, continuing this process,
we have that there are infinitely many vertices of type (v3.2). It contradicts the fact that $Q$ is finite.

Consequently, we conclude that $j$ is the unique sink in $Q$.
\end{proof}

Similarly, we have the following

\begin{proposition}\label{3.12}
Assume that there are no vertices of type (v4) in $Q$ and $j\in Q_0$ is a sink of type (v2.2).
Then the following statements are equivalent.
\begin{enumerate}
\item[(1)] The projective minimal right determiners are $\{P(i)\mid 1 \leq i \leq n$ but $i\neq j\}$.
\item[(2)] $j$ is the unique sink in $Q$.
\end{enumerate}
\end{proposition}

The following example illustrates that the assumption ``there are no vertices of type (v4) in $Q$"
is necessary for Propositions 3.11 and 3.12.

\begin{example}\label{3.13}
Let $Q$ be the quiver
$$
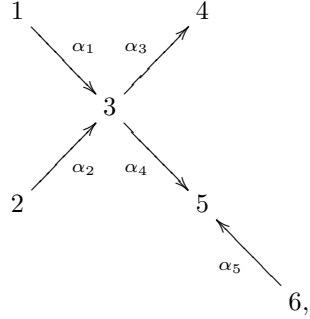
\xymatrix{
 1  \ar[dr]^{\alpha_1} & & 4 \\
 & 3 \ar[ur]^{\alpha_3} \ar[dr]_{\alpha_4}\\
 2  \ar[ur]_{\alpha_2} & &5  \\
 &&& 6,\ar[ul]^{\alpha_5}
}$$
and let $\Lambda=KQ/I$ such that at least one of the two sets $\{\alpha_3\alpha_1,\; \alpha_4\alpha_2\}$
and $\{\alpha_4\alpha_1,\; \alpha_3\alpha_2\}$ is in the admissible ideal $I$ of $KQ$ (that is, $\Lambda$ is a string algebra).
Then the projective minimal right determiners in $\mod \Lambda$ are $\{P(1),P(2),P(4),P(5),P(6)\}$ by Theorem 3.9.
But the vertex 4 is the unique sink of type (v1.2) and the vertex 5 is the unique sink of type (v2.2) in $Q$.
\end{example}

The following example illustrates that the assumption ``$j\in Q_0$ is a sink of type (v1.2)"
in Proposition 3.11 and the assumption ``$j\in Q_0$ is a sink of type (v2.2)" in Proposition 3.12 are necessary,
and that neither of the source counterparts of these two propositions holds true.

\begin{example}\label{3.14}
\begin{enumerate}
\item[]
\item[(1)] Let $Q$ be the quiver
$$\xymatrix{
1 \ar[r] &2 &3\ar[l]\ar[r] &4 \\
}$$
and $\Lambda=KQ$. Then the projective minimal right determiners in $\mod \Lambda$ are $\{P(1),P(2),P(4)\}$ by Theorem 3.9.

\item[(2)] Let $Q$ be the quiver
$$\xymatrix{
1  \\
& 3 \ar[ul]_{\alpha_1} \ar[dl]^{\alpha_2} & 4\ar[l]_{\alpha_3}\ar[r]^{\alpha_4} &5,\\
2  \\
}$$
and $\Lambda=KQ/I$ a bound quiver algebra. If $I$ is generated by $\{\alpha_1\alpha_3,\alpha_2\alpha_3\}$,
then by Theorem 3.9, the projective minimal right determiners are \linebreak
$\{P(1),P(2),P(3),P(5)\}$. If $I$ is generated
by $\{\alpha_1\alpha_3\}$, then by Theorem 3.9 again, the projective minimal right determiners are
$\{P(1),P(2),P(5)\}$.
\end{enumerate}
\end{example}

\section{Algebras of Dynkin type}

In this section, the quiver $Q$ is of Dynkin type
and $J_{i_1}, J_{i_2}, J_{i_3}$ are as in Definition 3.3. Note that there are no vertex ideals of type $J_{i_4}$ in this case.

It is trivial that if $\Lambda$ is of type $\mathbb{A}_n$, that is, the underlying 
graph of $Q$ is of the form
$$1\frac{\alpha_1}{}2 \frac{\alpha_2}{} 3\frac{\alpha_3}{}\cdots \frac{\alpha_{n-2}}{} n-1\frac{\alpha_{n-1}}{}n,$$
then $\Lambda$ is string. So by Theorem 3.10, we immediately have the following

\begin{corollary} \label{4.1}
If $\Lambda$ of type $\mathbb{A}_n$, then we have
$$|\Det(\Lambda)|=2n-p-q-1,$$
where $p=|\{i\mid i\ \text{is a source in}\ Q\ {\text with}\ 2\leq i\leq n-1 \}|$ and
$q=|\{J_{i_j}\neq 0 \mid j=1,2\}|$.
\end{corollary}

Let $\Lambda$ be of type $\mathbb{A}_n$. Then there are no vertex ideals of type $J_{i_3}$ or $J_{i_4}$.
\begin{enumerate}
\item[(1)] If there is a unique sink in $Q$, then we have the following facts.
\begin{enumerate}
\item[(1.1)] There are no sources $i$ with $2\leq i\leq n-1$.
\item[(1.2)] Either $J_{i_1}\neq 0$ or $J_{i_2}\neq 0$.
\end{enumerate}
So $p=0$ and $q=1$, and hence $|\Det(\Lambda)|=2n-2$ by Corollary 4.1.
\item[(2)] If $\Lambda$ is a path algebra with at least two sinks in $Q$, then $|\{J_{i_1}\neq 0\}|=0=|\{J_{i_2}\neq 0\}|$.
By Corollary 4.1, we have $|\Det(\Lambda)|=2n-p-1$. Thus [13, Theorem 3.13] follows.
\item[(3)] If $\Lambda$ is a bound quiver algebra with at least two sinks in $Q$, then it is straightforward to check
that the notion of vertex ideals is exactly that of sink ideals in [13, Definition 3.14]. By Corollary 4.1,
we have $|\Det(\Lambda)|=2n-p-q-1$. Thus [13, Theorem 3.15] follows.
\end{enumerate}
In conclusion, Corollary 4.1 is a unified version of [13, Theorems 3.13 and 3.15].

If $\Lambda$ is of type $\mathbb{D}_n$, then the underlying 
graph of $Q$ is of the form
$$
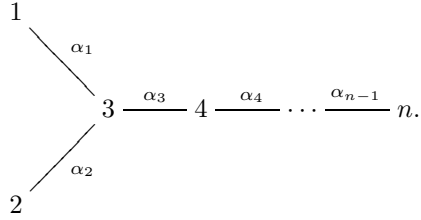
\xymatrix{
1 \ar@{-}[dr]^{\alpha_1} \\
& 3 \ar@{-}[r]^{\alpha_3} & 4 \ar@{-}[r]^{\alpha_4} & \cdots \ar@{-}[r]^{\alpha_{n-1}} & n. \\
2 \ar@{-}[ur]_{\alpha_2} \\
}$$
If $\Lambda$ is of type $\mathbb{E}_n$ with $6\leq n\leq 8$, then the underlying 
graph of $Q$ is of the form
$$
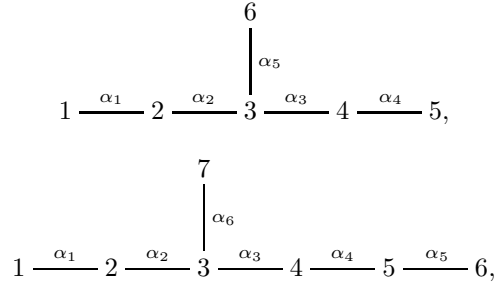
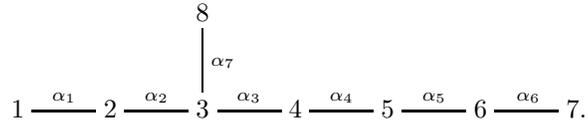
\xymatrix{
& & 6 \ar@{-}[d]^{\alpha_5}\\
1 \ar@{-}[r]^{\alpha_1} & 2 \ar@{-}[r]^{\alpha_2} & 3 \ar@{-}[r]^{\alpha_3} & 4 \ar@{-}[r]^{\alpha_4} & 5,
}$$
$$\xymatrix{
& & 7 \ar@{-}[d]^{\alpha_6}\\
1 \ar@{-}[r]^{\alpha_1} & 2 \ar@{-}[r]^{\alpha_2} & 3 \ar@{-}[r]^{\alpha_3}
& 4 \ar@{-}[r]^{\alpha_4} & 5 \ar@{-}[r]^{\alpha_5} &6,
}$$

or

$$\xymatrix{
& & 8 \ar@{-}[d]^{\alpha_7}\\
1 \ar@{-}[r]^{\alpha_1} & 2 \ar@{-}[r]^{\alpha_2} & 3 \ar@{-}[r]^{\alpha_3} & 4 \ar@{-}[r]^{\alpha_4}
& 5 \ar@{-}[r]^{\alpha_5} &6 \ar@{-}[r]^{\alpha_6} &7.
}$$

In the above four cases, we denote the subquivers in $Q$ with 4 vertices including the vertex
$3$ and its 3 neighbours by $X_{3}$. By the definition of string algebras, we have

\begin{proposition}\label{4.2}
Let $\Lambda$ be of type $\mathbb{D}_n$, $\mathbb{E}_6$, $\mathbb{E}_7$ or $\mathbb{E}_8$.
Then $\Lambda$ is a string algebra if and only if $\Lambda$ is a bound quiver algebra and $I|_{X_3}\neq 0$.
\end{proposition}

By Proposition 4.2 and Theorem 3.10, we have the following two corollaries.

\begin{corollary}\label{4.3}
Let $\Lambda$ be a bound quiver algebra of type $\mathbb{D}_n$ with $I|_{X_3}\neq 0$. Then we have
$$|\Det(\Lambda)|=2n-p-q-1,$$
where $p=|\{i\mid i\ \text{is a source in}\ Q \ \text{with}\ 4\leq i\leq n-1\}|$ and
$q=|\{J_{i_j}\neq 0 \mid j=1,2,3\}|$.
\end{corollary}

\begin{corollary}\label{4.4}
Let $\Lambda$ be a bound quiver algebra of type $\mathbb{E}_n$ with $6\leq n\leq 8$ and $I|_{X_3}\neq 0$. Then we have
$$|\Det(\Lambda)|=2n-p-q-1,$$
where $p=|\{i\mid i\ \text{is a source in}\ Q \ \text{with}\ i\neq 1,3,{n-1},n\}|$ and
$q=|\{J_{i_j}\neq 0 \mid j=1,2,3\}|$.
\end{corollary}

\section{An example of non-Dynkin type}

In this section, we give an example of non-Dynkin type to illustrate Theorem 3.10.

\begin{example}\label{5.1}
Let $Q^{(0)}$ be the quiver with a unique vertex but no arrows: $\circ$.
Let $Q^{(1)}$ be the quiver
$$
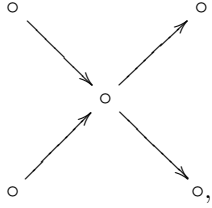
\xymatrix{
 \circ  \ar[dr] & & \circ  \\
 & \circ \ar[ur] \ar[dr]\\
 \circ  \ar[ur] & &\circ,  \\
}$$
and $Q^{(2)}$ the quiver:
$$
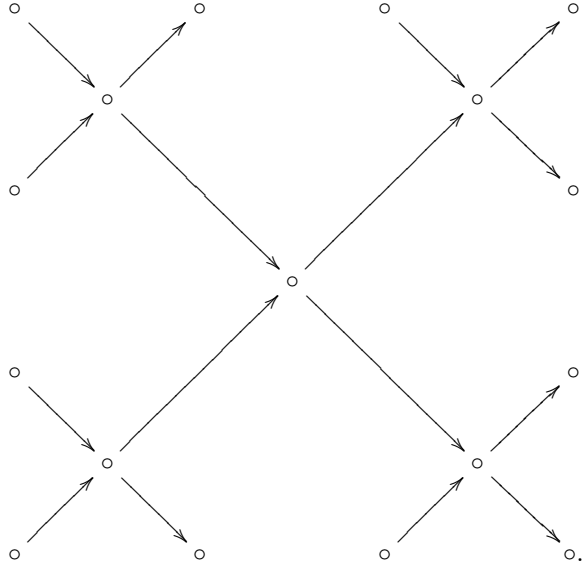
\xymatrix{
\circ  \ar[dr] & & \circ &  & \circ \ar[dr] & & \circ \\
& \circ \ar[ur] \ar[drdr] &&& &\circ \ar[ur]\ar[dr]  \\
\circ \ar[ur] &&&&& & \circ                          \\
&&  &\circ  \ar[urur]\ar[drdr]                      \\
\circ \ar[dr] &&&&& & \circ                          \\
& \circ \ar[urur] \ar[dr] &&& &\circ \ar[ur]\ar[dr]  \\
\circ  \ar[ur] & & \circ &  & \circ \ar[ur] & & \circ. \\
}$$
We call the following the {\bf first step}: $Q^{(1)}$ is a quiver of type $X_{i_4}$, which is obtained by adding 4 vertices of type (v1.1)
around the unique vertex in $Q^{(0)}$;
and call the following the {\bf second step}:
$Q^{(2)}$ is obtained by adding 3 vertices of type (v1.1) around each vertex of type (v1.1) in $Q^{(1)}$ such that
all the 4 new branches in $Q^{(2)}$ are of type $X_{i_4}$. Inductively, in the {\bf $n$-th step}, $Q^{(n)}$ is obtained from $Q^{(n-1)}$
by adding 3 new vertices of type (v1.1) around each of all $4\times 3^{n-2}$ vertices of type (v1.1) in $Q^{(n-1)}$ such that all the
$4 \times 3^{n-2}$ new branches are of type $X_{i_4}$. In $Q^{(n)}$, the number of vertices is $2\times 3^n-1$.

For any $n\geq 1$, let $\Lambda^{(n)}=KQ^{(n)}/I$ with $I$ the admissible ideal of $KQ^{(n)}$ generated by all the paths of length 2.
Notice that there are no vertices of type (v2) or (v3), so
$|\{J_{i_2}\neq 0\}|=0=|\{J_{i_3}\neq 0\}|$. It is easy to see that $|\{J_{i_1}\neq 0\}|=0$.

If $n\geq 2$, then we have that some $J_{i_4}\neq 0$ if and only if it is the vertex ideal of a vertex of type (v4) in the outermost ring
of $Q^{(n)}$. So $|\{J_{i_4}\neq 0\}|=4\times 3^{n-2}$, and hence by Theorem 3.10, for any $n\geq 2$ we have
$$|\Det(\Lambda^{(n)})|= 2(2\times 3^n-1)-0-4\times 3^{n-2}-1=32\times 3^{n-2}-3,$$
where the number of the projective minimal right determiners is $14\times 3^{n-2}-1$ and the number of the non-projective ones is $2\times 3^n-2$.
Moreover, by Theorem 3.9, we have that $P(i)\notin \Det(\Lambda^{(n)})$ if and only if
$i$ is one of the $4\times 3^{n-2}$ vertices added in the $(n-1)$-th step.

If $n=1$, then $|\{J_{i_4}\neq 0\}|=1$. So by Theorem 3.10, we have
$$|\Det(\Lambda^{(1)})|=2\times 5-0-1-1=8,$$
where both the number of the projective minimal right determiners and the number of the non-projective ones are 4.
By Theorem 3.9, we have that $P(i)\in \Det(\Lambda^{(1)})$ if and only if $i$ is one of the four vertices added in the first step.
\end{example}

\vspace{0.2cm}

{\bf Acknowledgement.} This research was partially supported by NSFC (Grant
No. 11571164) and a Project Funded by the Priority Academic Program Development of Jiangsu Higher Education Institutions.

\end{document}